\documentclass[12pt]{amsart}
\usepackage{amssymb}

\usepackage[T1]{fontenc}
\usepackage[utf8]{inputenc}
\usepackage[english]{babel}
\usepackage{amsfonts, amssymb, amsmath, amsthm}
\usepackage[all]{xy}
\usepackage{enumerate}

\usepackage{graphicx}
\usepackage{psfrag}
\usepackage{geometry}
\usepackage[colorlinks=true,pdfstartview=FitV,linkcolor=blue,citecolor=green,urlcolor=black,filecolor=magenta]{hyperref}
\usepackage{verbatim}

\newtheorem{theorem}{Theorem}[section]
\newtheorem{definition}[theorem]{Definition}

\newtheorem{lemma}[theorem]{Lemma}
\newtheorem{cor}[theorem]{Corollary}

\newcommand{\Hh}{{\mathcal H}}

\newcommand{\Ll}{{\mathcal L}}

\newcommand{\Pp}{{\mathcal P}}

\newcommand{\id}{{\mathrm{id}}}

\newcommand{\N}{\mathbb{N}}
\newcommand{\Z}{\mathbb{Z}}

\newcommand{\Xmax}{X_{max}}

\newcommand{\pmax}{\pi}
\newcommand{\pms}{\pmax_*}
\newcommand{\fb}[1]{\pmax^{-1}(#1)}

\newcommand{\supp}{\mbox{\rm supp}}
\newcommand{\ev}{ev_{0}}

\newcommand{\card}{\mathrm{card}}

\title{On non-tameness of the Ellis semigroup}
\author{Johannes Kellendonk}
\address{Universit\'e Claude Bernard Lyon 1, CNRS UMR 5208, Institut Camille Jordan, F-69622 Villeurbanne, France}
\email{kellendonk@math.univ-lyon1.fr}

\date{\today}

\begin{document}

\maketitle

\begin{abstract}
The Ellis semigroup of a dynamical system $(X,T)$ is tame 
if every element is the limit of a sequence (as opposed to a net) of homeomorphisms coming from the $T$ action. This topological property is related to the cardinality of the semigroup. Non-tame Ellis semigroups have a cardinality which is that of the power set of the continuum $2^{\mathfrak c}$.
The semigroup admits a minimal bilateral ideal and this ideal is a union of isomorphic copies of a group $\mathcal H$, the so-called structure group of $(X,T)$.  
For almost automorphic systems the cardinality of $\mathcal H$ is at most $\mathfrak c$, that of the continuum. We show a partial converse for minimal $(X,T)$ with abelian $T$, namely  that the cardinality of the structure group is $2^{\mathfrak c}$ if the proximal relation is not transitive and the subgroup generated by differences of singular points in the maximal equicontinuous factor is not open.
This refines the above statement about non-tame Ellis semigroups, as it locates a particular algebraic component of the latter which has such a large cardinality. 
\end{abstract}

\section{Introduction}

Let $F(X)$ be the set of functions from $X$ to $X$. Equipped with composition as multiplication and the topology of point-wise convergence $F(X)$ is a compact right topological semigroup. If $(X,T)$ is a topological dynamical system, by which we mean a compact space $X$ with an action of a group $T$ by homeomorphisms $\alpha^t$, $t\in T$, its Ellis or enveloping semigroup $E(X,T)$ is the closure of $\{\alpha^t:t\in T\}$  in $F(X)$. If no confusion arises we simply write $E$ or $E(X)$ for $E(X,T)$.
$E(X)$ is closed under composition 
and so a compact right topological sub-semigroup of $F(X)$. 

The Ellis semigroup has rich algebraic and topological properties, and these can be used to characterise the dynamical system. 
A recent survey on this can be found in \cite{glasner2007enveloping}.
One property which has attracted a lot of attention is tameness \cite{Kohler1995}. 
$E(X,T)$ (and $(X,T)$) is tame if all its elements are Baire class 1 functions, that is, can be obtained as a limit of a sequence of continuous functions \cite{glasnerCM}. An equivalent characterisation is that $E(X,T)$ is tame if its cardinality is at most $\mathfrak c$, that of the continuum  \cite{GlasnerMegrelishvili2006}. A third characterisation is that $E(X,T)$ is not tame if $X$ contains an independence sequence \cite{KerrLi2007}. 
 
Like any semigroup, $E(X,T)$ has an ideal structure and can be decomposed into the equivalence classes of the Green's relations. In the context of non-tame dynamical systems, it is interesting to know which of these parts are especially big. More specifically, we can look at the kernel $\ker E$ of $E(X,T)$ (its smallest bilateral ideal). Given any minimal idempotent $e\in E$ the kernel is the bilateral ideal generated by it, $\ker E = EeE$.  
It is a completely simple semigroup without zero and therefore has the following algebraic structure \cite{Howie,Hindman}.
Let $\Lambda$ be the set of minimal left ideals and $I$ the set of minimal right ideals of $E$, then
$$\ker E \cong I\times \Hh\times \Lambda$$
with multiplication
$$(i,g,\lambda) (j,h,\mu) := (i,g a_{\lambda \, j} h,\mu)$$
where $\Hh$ is the so-called {\em structure group} and $(a_{\lambda\, i})_{i\in I,\lambda\in\Lambda}$ is a matrix with values in $\Hh$. 
Given a minimal idempotent $e\in E$, $\Hh$ can be taken to be  
 $$ \Hh_e:= eEe$$
different choices of minimal idempotents leading to isomorphic groups. 

What can we say about the size of $I$, $\Hh$, and $\Lambda$? 
Let us mention some known results for minimal $(X,T)$.
\begin{enumerate}
\item $(X,T)$ is distal if and only if $E$ is a group \cite{Auslander,Hindman}.
 This implies that $E = \ker E = \Hh$ and $\card(I)=\card(\Lambda) = 1$.
\item If $(X,T)$ is equicontinuous with $T$ abelian so that we can equip $X$ with a group structure then 
$E = \ker E = \Hh=X$ \cite{Auslander,Hindman}, hence $\card(\Hh) \leq \mathfrak{c}$.
\item If $(X,T)$ is almost automorphic with maximal equicontinuous factor $\Xmax$ then
$\ker E \cong I\times E(\Xmax)$ with multiplication $(i,g)(j,h) = (i,gh)$ \cite{ABKL}, we recall the proof below. Thus we have $\card(\Lambda) = 1$ and $\Hh = E(\Xmax)$, and at least 
if $T$ abelian $\card(\Hh) \leq \mathfrak{c}$.
\item Dynamical systems arising from primitive aperiodic bijective substitutions provide examples of non-tame Ellis semigroups for which $E = \Z\cup \ker E$ with finite $I$ and $\Lambda$ but $\card(\Hh) = 2^{\mathfrak{c}}$ 
\cite{KY}.
\item If $X$ is metrisable and $\card{\Lambda}\leq \mathfrak{c}$ then $(X,T)$ is a PI-flow \cite{GG20}.
\end{enumerate}
To illuminate the last result we recall that  
$(X,T)$ is a PI-flow if it admits a proximal extension $(\tilde X,T)$ which itself is a tower of PI-extensions of the one-point system $(pt,T)$, and a PI-extension is an extension which is a composition of a proximal with an isometric extension \cite{GG20}. As is also shown in \cite{GG20}, the converse is not always the case, there exists a PI minimal metrisable system $(X,T)$ even with $T=\Z$ for which 
$E(X, T)$ contains nevertheless $2^\mathfrak{c}$ minimal left ideals.

In this work we focus on the cardinality the structure group $\Hh$. Starting point is the observation that this group need not to have the same cardinality as the Ellis semigroup. Indeed, there are almost automorphic non-tame dynamical systems with $T=\Z$ and hence their structure group coincides with their maximal equicontinuous factor and so has a cardinality $\leq \mathfrak c$. To state our result we need to provide more background.

Recall that $(X,T)$ is equicontinuous if the family of homeomorphisms $\{\alpha^t:t\in T\}$ is equicontinuous. If $(X,T)$ is moreover minimal and $T$ abelian then $X$ carries a group structure (which we denote additively) and a unique $T$-invariant probability measure, namely the Haar measure.
Any dynamical system $(X,T)$, equicontinuous or not, admits a maximal equicontinuous factor $\pi:(X,T)\to (X_{max},T)$. 
The equivalence relation $x\sim y$ iff $\pi(x)=\pi(y)$ is called the equicontinuous structure relation. 
 A point $\xi\in X_{max}$ is called singular, if the fibre $\pi^{-1}(\xi)$ contains two proximal points.  
Otherwise it is called regular. A system is called almost automorphic if it is minimal and there is $\xi\in X_{max}$ such that $\pi^{-1}(\xi)$ contains a single point. 

A recent result states that if a minimal system with abelian $T$ is tame then it is almost automorphic and the set of regular points 
has full Haar measure in $X_{max}$ \cite{fuhrmann2018irregular,glasner2018structure}. 
The converse need not to be true. Toeplitz systems (minimal almost automorphic extensions of odometers) need not be tame, although their set of regular points has full measure. For instance, of the two symbolic dynamical systems associated to the two substitutions
$$\begin{matrix}
a & \mapsto & a a b a a\\
b & \mapsto & a b b a a
\end{matrix} \quad \qquad
\begin{matrix}
a & \mapsto & a a b a a\\
b & \mapsto & a b a b a
\end{matrix}
$$
the one on the left is tame, whereas the other is not tame \cite{FKY}. 
Note that these substitutions differ only in the order of two letters, in particular their associated dynamical systems are strong orbit equivalent.

The following result was stated in \cite{ABKL} for abelian $T$. The proof given there extends verbatim to non-abelian $T$.
\begin{theorem}
Let $(X,T)$ be a minimal system. If the proximal relation agrees with the equicontinuous structure relation then the structure group is isomorphic to the Ellis semigroup of the maximal equicontinuous factor of $\Xmax$. The isomorphism is also a homeomorphism.
\end{theorem}
We note that for almost automorphic systems the proximal relation always agrees with the equicontinuous structure relation, but, as we already said above, this does not imply that $E$ is tame. 

As partial converse to the above we obtain the following statement.
\begin{theorem}
Let $(X,T)$ be a minimal system with abelian $T$. Let $\Ll^{sing}$ be the group generated by differences $\xi-\xi'$ of singular points $\xi,\xi'$.
If the proximal relation is not transitive and 
 $\Xmax/\Ll^{sing}$ uncountable then the structure group has cardinality $2^\mathfrak{c}$.
\end{theorem}
The condition that $\Xmax/\Ll^{sing}$ is uncountable is equivalent to $\Ll^{sing}$ being open.\footnote{We thank Todor Tsankov for explaining us this equivalence.} 
Note that this condition is always satisfied if there are only countably many singular points in $X_{max}$ while $X_{max}$ is uncountable. 
We prove this theorem using a method already employed in \cite{KY}.

\section{Preliminaries}
We recall here some concepts and results for minimal topological dynamical systems $(X,T)$ with compact metrisable space $X$. 

Two points $x,y\in X$ are proximal, written $x\sim_p y$, if $\inf_{t\in T} d(\alpha^t(x),\alpha^t(y)) = 0$ where $d$ is a metric which induces the topology and $\alpha$ the $T$-action. This notion does not depend on the choice of metric. A point is distal if it is not proximal with any other point. $(X,T)$ is point distal if it is minimal and contains a distal point.

We denote by $(\Xmax,T)$ the maximal equicontinuous factor of $(X,T)$ (which is uniquely determined up to isomorphism) and its factor map by $\pmax$. The equicontinuous structure relation on $X$ is the relation induced by $\pmax$: $x\sim y$ iff $\pmax(x)=\pmax(y)$. It is an equivalence relation which always contains the proximal relation.

We say that $(X,T)$ has a {\em finite distal fibre} if there is $\xi\in \Xmax$ such that $\pmax^{-1}(\xi)$ is finite and its points are pairwise non-proximal. 

We define the {\em coincidence rank} of the fibre $\pmax^{-1}(\xi)$ to be 
$$cr(\xi):=\sup\{l\in\N:\exists x_1,\dots,x_l\in\pmax^{-1}(\xi),x_i\not\sim_p x_l\}$$
By Lemma~2.10 \cite{BK} $cr(\xi)$ is the same for all $\xi$ if $(X,T)$ is minimal\footnote{While \cite{BK} makes the standing assumption that $T$ is abelian, this is not required for the proof of Lemma~2.10}
We therefore may call $cr=cr(\xi)$ the coincidence rank of the minimal system $(X,T)$. 
\begin{lemma} Let $(X,T)$ be minimal with metrisable $X$.
\begin{enumerate}
\item
$(X,T)$ has a finite distal fibre if and only if it is point distal and has finite coincidence rank. 
\item $cr=1$ if and only if the proximal relation agrees with the equicontinuous structure relation. In particular, in this case the proximal relation is transitive.
\item If $cr$ finite and $T$ contains a compact set $K$ such that any open set containing $K$ generates $T$, then transitivity of the proximal relation implies that $cr=1$. 
\item
If $(X,T)$ is almost automorphic then $cr=1$.
\item
If $(X,T)$ is point distal and $cr=1$ then it is almost automorphic.
\end{enumerate}
\end{lemma}
\begin{proof}
1. 
The implication $"\Rightarrow"$ is direct. For the other direction $"\Leftarrow"$ suppose that $(X,T)$ is point distal. By a result of Ellis the distal points are then residual \cite{Ellis}. By a result of Veech there is $\xi$ such that $\pmax^{-1}(\xi)$ contains a dense set of distal points \cite{Veech}. As $cr(\xi)$ is finite, $\pmax^{-1}(\xi)$ contains finitely many distal points. The closure of a finite set being finite we conclude that   $\pmax^{-1}(\xi)$ is a finite distal fibre.

2. If $cr=1$ then all points of a fibre belong to the same proximal class. As the proximal relation is contained in the equicontinuous structure relation if follows that they must coincide.

3. Suppose that the proximal relation $P$ is transitive. This implies that $P$ is topologically closed \cite{Auslander} and hence $X/P$ is compact in the induced metric. As $cr$ is finite by hypothesis there exists $\delta_0$ such that distal points of $X$ which. belong to the same fibre $\pmax^{-1}(\xi)$ have distance at least $\delta$. Hence $(X/P,T)$ is an equicontinuous extension of $(\Xmax,T)$. As shown in \cite{SackerSell}, the condition on $T$ implies that $(X/P,T)$ is an equicontinuous system. Hence $X/P=\Xmax$, hence $cr=1$.

4. is direct, as we can measure cr at the point $\xi$ whose fibre is a singleton.

5. This follows from 1., because point distal and $cr=1$ imply that there is a finite distal fibre, say at $\xi$, and since $cr(\xi)=1$ this fibre has a single point.
\end{proof}

\section{The structure of the kernel of $E$}
If a semigroup admits a smallest bilateral ideal this ideal is called the kernel of the semigroup.  
Compact right-topological semigroups admit always a kernel. This kernel, which we denote $\ker E$, is a completely simple semigroup (without zero element) whose structure we now partly describe. See \cite{Hindman} or \cite{Howie} for details.

Let $J_{min}$ be the set of idempotents in $\ker E$. These are called minimal idempotents.
$\ker E$ is the bilateral ideal generated by $p$ for any choice of $p\in J_{min}$. $\ker E$ is partitioned by its left ideals and two idempotents $p,q$ of the same left ideal satisfy $pq=p$.  $\ker E$ is also partitioned by its right ideals  and two idempotents $p,q$ of the same right ideal satisfy $pq=q$. The intersection of a left ideal with a right ideal contains a unique idempotent $p$ and is a group, namely $\Hh_p:=p E p$. All these groups for different choices of $p$ are isomorphic. If $p,q$ belong to the same left ideal then the isomorphism $\Hh_p\to\Hh_q$ is given by left multiplication with $p$, and if $p,q$ belong to the same right ideal then the isomorphism $\Hh_p\to\Hh_q$ is given by right multiplication with $q$. In particular, given any two  $p,q\in J_{min}$ there are $p',q'\in J_{min}$ such that $x\mapsto p'xq'$ is a group isomorphism from $\Hh_p$ to $\Hh_q$.
We should mention that the isomorphisms above are not all homeomorphisms and the groups not all homeomorphic, but this will not be an issue for what we do. 

Let $\Gamma_p$ be the group generated by $pJ_{min} p$. As the isomorphism between $\Hh_p$ and $\Hh_q$ is given by multiplication with idempotents, the same isomorphism maps $\Gamma_p$ to $\Gamma_q$. 
\begin{definition}
Let $e$ be a minimal idempotent. The {\em structure group} is the group $\Hh_e=eEe$. The {\em little structure group} $\Gamma_e$ is its subgroup generated by $eJ_{min} e$ 
\end{definition}
We recall that $\Hh_e$ depends on the choice of $e$ only up to isomorphism.  
In the general theory of semigroups $\Hh_e$ is called the Rees-structure group of $\ker E$.

\subsection{When the structure group is small}

Like any factor map, $\pmax:X\to \Xmax$ induces an epimorphism of semigroups $\pmax_*:E(X)\to E(\Xmax)$, namely $\pmax_*(f)(\xi) = \pmax(f(x))$ where $x$ is any element of $\fb{\xi}$. 
Note that idempotents $p\in E$ preserve the fibres of $\pmax$ and hence ${\pmax}_*(p)(\xi) = \pmax(p(x))=\xi$. It follows that ${\pmax}_*:E(X)\to E(\Xmax)$ restricts to a (continuous) epimorphism of groups
$$\left.{\pmax}_*\right|_{\Hh_e} : \Hh_e\to E(\Xmax)$$
whose kernel contains $\Gamma_e$. If no confusion is possible, we write  $\pms$ for $\left.{\pmax}_*\right|_{\Hh_e}$.
\begin{lemma}[\cite{ABKL}]
Let $(X,T)$ be minimal.
If the proximal relation agrees with the equicontinuous structure relation ($cr=1$) then  $\pms: \Hh_e\to E(\Xmax)$ is an isomorphism. If moreover $T$ contains a compact set $K$ such that any open set containing $K$ generates $T$ then the converse is true as well,
 $\pms$ is an isomorphism only if $cr=1$.
\end{lemma}
While the above isomorphism is a continuous bijection, it is not bi-continuous, as $eEe$ is not compact if $E$ contains more than one minimal idempotent. 
\begin{proof}
The proof of $1 \Rightarrow 2$ is as in Lemma~5.4 of \cite{ABKL} and works also for non-abelian $T$. It does not require $T$ to contain a compact set $K$ such that any open set containing $K$ generates $T$. Let's recall it: Let $f\in eEe$ with ${\pmax}_*(f) = \id$. Then $f$ preserves the fibres of $\pmax$. As $c=1$ all elements of a fibre are proximal. Hence $x\sim_p f(x)$ for all $x$. Hence there is an idempotent $p$ in the unique minimal left ideal such that $p(x) = pf(x)$. As also $e$ is in that minimal left ideal we have $ep = e$. This implies $e(x) = f(x)$, hence $f=e$. 

As for $2 \Rightarrow 1$ suppose that $cr>1$. As $T$ contains a compact set $K$ such that any open set containing $K$ generates $T$ this implies that the proximal relation is not transitive hence $\Gamma$ is not trivial. 
Hence $\ker{\pmax}_*$ is not trivial.
\end{proof}
We recall that $\Xmax$ is isomorphic to the quotient  
$E(\Xmax)/F_z$ by a subgroup of elements which have a given fixed point $\xi$. If $T$ is abelian, then there are no such elements.
\begin{cor}
If $E(\Xmax)$ acts fixed point freely then $cr=1$ implies that the structure group is small.
\end{cor}

\section{The structure group for systems with non-transitive proximal relation}
In this section we show that, for minimal systems with abelian group action, non-triviality of the little structure group together with an extra condition on the set of singular points implies that the structure group has cardinatlity $2^\mathfrak{c}$. This is a partial converse of the results of the last sections where we showed that $cr=1$ implies that the structure group is isomorphic to $E(X_{max})$, a small group at least if its subgroup of elements which have a fixed point is small. Furthermore, we saw that $cr=1$ is implied if the system is almost automorphic, or, if  the little structure group is trivial and $T$ contains a compact set $K$ such that any open set containing $K$ generates $T$.

\subsection{The proximal relation and the little structure group}
As is well known, the proximal relation can be studied with the help of the Ellis semigroup. Indeed, $x,y\in X$ are proximal if and only if there exists a minimal idempotent $p\in E$ such that $p(x)=p(y)$. If the system is minimal then this can be strengthened: 
\begin{lemma}[\cite{Auslander} p.~89, Thm.~13(iii))]
Let $(X,T)$ be minimal. Then $x,y\in X$ are proximal if and only if there exists a minimal idempotent $p$ such that $p(y)=x$.
\end{lemma}
\begin{proof} We first show that for any $x\in X$ one can find a minimal idempotent $p\in E$ such that $p(x) = x$. Let $x\in X$. Let $q$ be any minimal idempotent and $y=q(x)$. By minimality its $T$-orbit is dense in $X$ hence $E(y)=X$. Thus there is $f\in E$ such that $x=f(y)= fq(x)$. As $q$ is minimal, $fq\in \ker E$. Upon replacing $f$ by $fq$ we may thus assume $x=f(x)$ for some $f\in\ker E$. By the Rees structure theorem
$f$ belongs to one of the groups $\Hh_{p}$, $p\in J_{min}$. Let $f^{-1}$ be its inverse in this group. Then $p =   f f^{-1}$ is a minimal idempotent and $p(x) =  ff^{-1}(x) = ff^{-1}f(x) = f(x) = x$.
  
Let $x,y\in X$ be proximal. Let $p$ be a minimal idempotent such that $p(x) = x$.
Let $q$ be a minimal idempotent such that $q(x)=q(y)$. Let $r$ be the (unique) minimal idempotent in the intersection of the minimal right ideal to which $p$ belongs with the minimal left ideal to which $q$ belongs.
Then $rp=p$ and $rq=r$. Hence
$$x = p(x) = rp(x) = rqp(x) = rq(x)=rq(y) = r(y).$$
As for the converse, $p(y) = x$ implies $p(y)=p(x)$ hence $x$ and $y$ are proximal.
\end{proof}
It is well-known that the proximal relation is transitive if and only if $E$ has a unique minimal left ideal, that is $\ker E$ is left simple. We can relate this also to the triviality of the little structure group.

Note that the little structure group is trivial if and only if the product of two idempotents is an idempotent. Indeed, the only idempotent in $pJ_{min} p$ is $p$. Moreover, if the product $pq$ of two minimal idempotents $p,q\in J_{min}$ is not an idempotent and $p'$ is such that $pq\in\Hh_{p'}$ (such a $p'$ must exist by the above) then $pq=p'pqp'\neq p'$ hence $p'J_{min} p'$ is not trivial.
Semigroups whose idempotents form a subsemigroup are called orthodox. So $\Gamma$ is trivial if and only if $\ker E$ is orthodox.

Let $(X,T)$ be minimal with maximal equicontinuous factor $\pi:X\to \Xmax$. Let $\xi\in \Xmax$. 
\begin{lemma} \label{lem-trans}
Let $(X,T)$ be minimal with maximal equicontinuous factor $\pi:X\to \Xmax$. Let $\xi\in \Xmax$. The proximal relation restricted to $\pi^{-1}(\xi)$ is transitive if and only if $\Gamma_e$ acts trivially on $e\fb{\xi}$. In particular, the proximal relation is transitive if and only if 
the little structure group $\Gamma_e$ is trivial.
\end{lemma}
\begin{proof} Suppose that the proximal relation restricted to $\pi^{-1}(\xi)$ is transitive and hence an equivalence relation. Let $\Pp$ be the associated partition of $\pi^{-1}(\xi)$ and $M$ a member of $\Pp$. Let $p,q$ be minimal idempotents of $E$. Then, as $p(x)$ is proximal to $x$, $p(x)\in M$ for all $x\in M$. This implies $pq(x)=p(x)$ for all $x\in\pi^{-1}(\xi)$. Hence the product of two minimal idempotents acts as an idempotent on $\pi^{-1}(\xi)$. Therefore $\Gamma_e$ acts trivially on $e\fb{\xi}$.

Suppose that $\Gamma_e$ acts trivially on $e\fb{\xi}$.
Let $(x,y)$ and $(y,z)$ be proximal pairs in $\pi^{-1}(\xi)$. As the system is minimal there exist minimal idempotents $p,q$ s.th.\ $p(x)=y$ and $q(y)=z$. Hence $qp(x) = z$. By assumption, $qp$ acts as an idempotent on $\pi^{-1}(\xi)$. Hence $qp(z) = qp(x)$ showing that 
$(x,z)$ is a proximal pair. 

The last statement follows as $\Gamma_e$ acts trivially on $e\pi^{-1}(\xi)$ for all $\xi$ if and only if $\Gamma_e=\{e\}$.
\end{proof}

\subsection{When the structure group is huge}
We now consider minimal systems for which the proximal relation is not transitive and hence $\Gamma_e$ not trivial. This excludes finite systems and so we assume now that $X$ is uncountable. 
Define $K_e\subset \Hh_e$ to be the subgroup of elements which preserve the fibres of $\pmax$. Clearly $\Gamma_e$ is a subgroup of $K_e$. 
We assume that $T$ is abelian and hence $\Xmax$ admits a group structure which we denote additively and its neutral element by $0$. It is well known that $E(\Xmax)$ is a group which is isomorphic to $\Xmax$, an isomorphism is given by $\ev$, the evaluation at $0$. 
We have an exact sequence of groups 
$$K_e\hookrightarrow \Hh_e \stackrel{\ev \circ \pms}\twoheadrightarrow \Xmax$$
where $\pi_*:E(X)\to E(\Xmax)$ is the epimorphism induced by $\pi$.
Choose a right inverse
$s:\Xmax\to \Hh_e$ to $\ev \circ \pms$ which intertwines the $T$ actions, i.e,\ 
 $\ev\circ\pmax_*\circ s = \id$ and $s\circ \delta^t = \alpha^t\circ s$ ($s$ need not be a group homomorphism). 

Let $f\in \Hh_e$. We say that $f$ acts trivially at $\xi\in\Xmax$ if all points of $e\fb{\xi}$ are fixed points of $f$. We
define the support of $f$ to be the set of points at which $f$ acts non-trivially. 
Recall that a point $\xi\in \Xmax$ is called regular if $\fb{\xi}$  contains only distal points and denote by $\Xmax^{sing}$ the complement of the regular points.
\begin{lemma}
Let $e\neq f\in \Gamma_e$. 
Then $\emptyset\neq \mathrm{supp}(f) \subset \Xmax^{sing}$.
\end{lemma}
\begin{proof} Direct consequence of (\ref{lem-trans}).
\end{proof}
\begin{lemma}\label{lem-shift}
Let $e\neq f\in K_e$ and $a\in\Xmax$. The support of $s(a) f s(a)^{-1}$ is $\supp(f)+a$. 
\end{lemma}
\begin{proof}
Let $e\neq f\in K_e$. 
Let $\xi\in \supp(f)+a$, that is, $\xi-a\in \supp(f)$. This is the case iff $\exists x\in e\fb{\xi-a}:f(x)\neq x$. This is equivalent to 
$\exists x'\in e\fb{\xi}:s(a)fs(a)^{-1}(x') \neq x'$ which, by definition, means $\xi\in \supp(s(a)fs(a)^{-1})$.
%
\end{proof}
\begin{theorem}
Let $(X,T)$ be a minimal system with abelian $T$.
If there is $e\neq f\in K_e$ and an uncountable set $A\subset X_{max}$ such that for all $a\neq b\in A$ we have $$(\supp(f)+a)\cap (\supp(f)+b) = \emptyset$$ then the Ellis group has cardinality $2^\mathfrak{c}$. 
\end{theorem}
\begin{proof}
Let $f$ and $A$ as in the theorem. As $f\neq e$ its support is not empty. Given $a\in A$ set $f_a=s(a)f
s(a)^{-1}$. By Lemma~\ref{lem-shift} we have $\supp(f_a) = \supp(f)+a$. Let $a_1,\cdots,a_k$ be $k$ distinct points of $A$ and $f_{\{a_1,\cdots,a_k\}} = f_{a_1} \cdots f_{a_k}$, the product. As 
$(\supp(f)+a_i)\cap (\supp(f)+a_j) = \emptyset$ for $i\neq j$ we have
$$\supp(f_{\{a_1,\cdots,a_k\}}) = \bigcup_{i=1}^k (\supp(f)+a_i).$$
Let $B\subset A$ and consider the directed system of all finite subsets $F_\nu\subset B$ of $B$, ordered by inclusion. The the net $f_{F_\nu}$ converges in the pointwise topology to the element $f_B$ which, at $\xi\in \supp(f)+b$, $b\in B$ acts as $f_b$ while it acts trivially at all points in the complement of $\supp(f)+B$. As $Ee$ is closed, $f_B\in Ee$. Moreover, for all $\xi\in\Xmax$ and $x\in\fb{\xi}$ the net $f_{F_\nu}(x)$ becomes eventually constant. Therefore, and since  $f_{F_\nu}(x)\in  e\fb{\xi}$, we must have $ef_B=f_B$ hence $f_B\in\Hh_e$. By construction, $f_B\neq f_{B'}$ if $B$ and $B'$ are different subsets of $A$. Hence $B\mapsto f_B$ is an injective map from the power set of $A$ to $\Hh_e$.
\end{proof}

Given $f\in K_e$ let $\Ll_f$ be the subgroup of $\Xmax$ generated by differences $\xi-\xi'\in \supp(f)-\supp(f)$.
\begin{cor}
Let $(X,T)$ be a minimal system with abelian $T$. 
If there is $f\in K_e$ such that 
$\Ll^f$ is not empty and $\Xmax/\Ll_f$ is uncoutable then the structure group has cardinatlity $2^\mathfrak{c}$.
\end{cor}
\begin{proof}
Let  $\tilde s: X_{max}/\Ll_f \to X_{max}$ be a right inverse of the quotient map, i.e.\ $s(\tilde a)$ is a choice of representative for $\tilde a$. Set $A=s(X_{max}/\Ll_f)$ and let  $a, b\in A$. 
If $(\supp(f)+a)\cap (\supp(f)+b) \neq\emptyset$ then $a-b\in \supp(f)-\supp(f)$ hence $a -b \in \Ll_f$. By definition of $A$ this implies $a=b$ and so $A$ satisfies the condition of the last theorem.
\end{proof}
Finally we formulate a criterion which is perhaps easiest checked. Note that, if $\Ll^{sing}$ is open then $\Xmax/\Ll^{sing}$ is finite, by compactness of $\Xmax$. 
\begin{cor}
Let $(X,T)$ be a minimal point distal system with abelian $T$. Suppose that the proximal relation is not transitive. If 
the group $\Ll^{sing}$ generated by differences 
$\xi-\xi'\in \Xmax^{sing}-\Xmax^{sing}$ is not open then the structure group has cardinatlity $2^\mathfrak{c}$.
\end{cor}
\begin{proof} 
If the proximal relation is not transitive there is $f\in\Gamma_e$ with $\emptyset\neq \supp f \subset \Xmax^{sing}$ hence $\emptyset\neq \Ll_f\subset \Ll^{sing}$. 
Suppose that $\Xmax/\Ll^{sing}$ is countable.  Then  $\Ll^{sing}$ cannot be meager. 
As the regular points are residual, $\Xmax^{sing}$ and therfore also $\Ll^{sing}$ is a Borel set, hence analytic. It therefore has the Baire property \cite{Kechris}[Thm.29.5]. By Pettis' Theorem \cite{Kechris}[Thm.~9.9] $\Ll^{sing}$ must be open. 
As this contradicts our assumption $\Xmax/\Ll^{sing}$ must be uncountable. Hence also $\Xmax/\Ll^{f}$ is uncountable and the result follows from the last corollary.
\end{proof}
If there are only  countably many singular points in $X_{max}$ then 
$\Ll^{sing}$ is, of course, not open and thus the structure group has cardinatlity $2^\mathfrak{c}$. 
 If there is only one singular orbit then more can be said about $\Hh_e$, see \cite{KY}. 
For instance, if $(X_\theta,{\mathbb Z})$ is the shift dynamical system associated to a bijective substitution $\theta$ of length $\ell$ with trivial generalised height then the maximal equicontinuous factor is the $\ell$-adic odometer $\mathbb Z_\ell$, and the structure group $\Hh$ 
topologically and algebraically isomorphic to 
the group of all functions $g:\mathbb Z_\ell/{\mathbb Z}\to G_\theta$
from the space of orbits of $\mathbb Z_\ell$  
to the group $G_\theta$ which generated by the column maps of the substitution. Here the group multiplication on the group of functions is point-wise and the topology is that of point-wise convergence. 
The fact that any function is allowed leads to the large cardinality.

\end{document}